\documentclass[11pt]{amsart}
\usepackage{amscd,amsmath,amssymb,amsfonts,verbatim}
\usepackage[cmtip, all]{xy}

\newtheorem{thm}{Theorem}[section]

\newtheorem{lem}[thm]{Lemma}
\newtheorem{cor}[thm]{Corollary}

\newtheorem{ques}[thm]{Question}
\theoremstyle{definition}

\newtheorem{defn}[thm]{Definition}
\theoremstyle{remark}
\newtheorem{remk}[thm]{Remark}
\newtheorem{remks}[thm]{Remarks}

\newtheorem{exm}[thm]{Example}
\newtheorem{exms}[thm]{Examples}
\newtheorem{notat}[thm]{Notation}
\numberwithin{equation}{section}

{\hfill$\square$\end{defn}}
{\hfill$\square$\end{remk}}
{\hfill$\square$\end{remks}}
{\hfill$\square$\end{exm}}
{\hfill$\square$\end{exms}}
{\hfill$\square$\end{notat}}

\newcommand{\thmref}{Theorem~\ref}

\newcommand{\lemref}{Lemma~\ref}

\newcommand{\sC}{{\mathcal C}}

\newcommand{\sF}{{\mathcal F}}

\newcommand{\sH}{{\mathcal H}}

\newcommand{\sK}{{\mathcal K}}
\newcommand{\sL}{{\mathcal L}}
\newcommand{\sM}{{\mathcal M}}

\newcommand{\sO}{{\mathcal O}}

\newcommand{\sR}{{\mathcal R}}

\newcommand{\sZ}{{\mathcal Z}}
\newcommand{\A}{{\mathbb A}}

\newcommand{\G}{{\mathbb G}}
\renewcommand{\H}{{\mathbb H}}

\renewcommand{\P}{{\mathbb P}}
\newcommand{\Q}{{\mathbb Q}}

\newcommand{\Z}{{\mathbb Z}}

\newcommand{\fm}{{\mathfrak m}}

\newcommand{\fp}{{\mathfrak p}}

\newcommand{\CH}{{\rm CH}}

\newcommand{\surj}{\twoheadrightarrow}
\newcommand{\inj}{\hookrightarrow}

\newcommand{\Pic}{{\rm Pic}}

\newcommand{\Spec}{{\rm Spec \,}}

\newcommand{\Sch}{{\operatorname{\mathbf{Sch}}}}

\newcommand{\Sm}{{\mathbf{Sm}}}

\newcommand{\ds}{{/\kern-3pt/}}

\newcommand{\ov}{\overline}

\newcommand{\tuborg}{\left\{\begin{array}{ll}}
\newcommand{\sluttuborg}{\end{array}\right.}

\newcommand{\wt}{\widetilde}

\begin{document}
\title{On 0-cycles with modulus}
%{Bloch's formula for cycles with modulus and applications}
\author{Amalendu Krishna}
\address{School of Mathematics, Tata Institute of Fundamental Research,  
1 Homi Bhabha Road, Colaba, Mumbai, India}
\email{amal@math.tifr.res.in}
%\address{Department of Mathematical Sciences, KAIST, 291 Daehak-ro Yuseong-gu, 
%Daejeon, 305-701, Republic of Korea (South)}
%\email{jinhyun@mathsci.kaist.ac.kr; jinhyun@kaist.edu}



\maketitle

%\tableofcontents

\begin{abstract}
Given a non-singular surface $X$ over a field and an effective Cartier divisor
$D$, we provide an exact sequence connecting $\CH_0(X,D)$ and the
relative $K$-group $K_0(X,D)$. We use this exact sequence to answer a
question of Kerz and Saito whenever $X$ is a resolution of singularities
of a normal surface. This exact sequence and two vanishing theorems are used  
to show that the localization sequence for ordinary Chow groups does not 
extend to Chow groups with modulus. This in turn shows that the additive
Chow groups of 0-cycles on smooth projective schemes can not always be 
represented as reciprocity functors.
\end{abstract}

\section{Introduction}\label{sect:Intro}
The idea of algebraic cycles with modulus was first conceived by
Bloch and Esnault (see \cite{BE-1} and \cite{BE-2}). One main motivation
behind such a theory is to develop a theory of motivic cohomology which
can describe the relative $K$-theory of smooth schemes relative to 
closed subschemes. 
A potential candidate for such a theory was later
constructed and studied by Park \cite{P1}, Krishna-Levine \cite{KL} and
more recently by
Kerz-Saito \cite{KeS} and Binda-Saito \cite{BS}. It was conjectured in
\cite{KL} that there should exist a spectral sequence consisting of these
motivic cohomology groups whose abutment is the relative $K$-theory.
%After \cite{KeS}, the elements of these motivic cohomology groups are
%usually called the {\sl cycles with modulus}.

\begin{comment}
It was shown by Bloch-Esnault \cite{BE-2} and R\"ulling \cite{Ruelling}
that certain 0-cycles with modulus on the affine line do compute a part
of the relative $K$-theory. Since then, there has not been much progress on this
connection between the cycles with modulus and the relative $K$-theory.
In a recent work \cite{BiK} (see also \cite{Binda} and \cite[\S~3]{Kr-1}), 
a cycle class map was constructed from the Chow group of 0-cycles with modulus 
to a relative $K$-group. 
\end{comment}

The results of this text were partly motivated by the following question of 
Kerz and Saito (see \cite[Question~V]{KeS}). Let $X$ be a smooth 
quasi-projective scheme of dimension $d$ over a field $k$ and let 
$D \inj X$ be an effective Cartier divisor. Let $\CH_0(X,D)$ denote the Chow 
group 0-cycles on $X$ with modulus $D$. Let $\sK^M_{d, (X,D)}$ denote the 
relative Milnor $K$-theory sheaf on $X$. Let $U$ be an open subscheme of $X$ 
whose complement is a divisor.

\begin{ques}\label{ques:KES-1}
Assume that $X$ is projective and $k$ is a perfect field of positive
characteristic. Is there an isomorphism
\[
{\underset{D}\varprojlim} \ \CH_0(X,D) \xrightarrow{\simeq}
{\underset{D}\varprojlim} \ H^d_{nis}(X, \sK^M_{d, (X,D)}),
\]
where the limits are taken over all effective divisors on $X$ with
support outside $U$?
\end{ques}

It follows from the main results of \cite{KS}, \cite{KeS} and \cite{RS} that 
this question has a positive solution if $k$ is finite and 
the support of $X \setminus U$ is a normal crossing divisor. As explained in
\cite{KeS}, the above question is part of the bigger question of whether
the Chow groups with modulus satisfy Nisnevich or Zariski descent.
As we shall see shortly, the above question is also directly related to
the conjectured connection between the cycles with modulus and the relative 
$K$-theory.  

\subsection{Main results}\label{sect:MR}
Let $\Pic(X,D)$ denote the isomorphism classes of pairs $(\sL, \phi)$, where 
$\sL$ is a line bundle on $X$ and $\phi$ is an isomorphism 
$\phi: \sL|_D \xrightarrow{\simeq} \sO_D$. We prove the following result
as a partial answer answer to the above question.

\begin{thm}\label{thm:Surj-FF}
Let $k$ be any field and let $X$ be a non-singular quasi-projective surface 
over $k$ with an effective Cartier divisor $D$. 
Then there is an exact sequence
\begin{equation}\label{eqn:Surj-FF-0}
\CH_0(X,D) \xrightarrow{cyc_{(X,D)}} K_0(X,D) \to \Pic(X,D) \to 0.
\end{equation}
In particular, $cyc_{(X,D)}$ induces a surjective map
$\CH_0(X,D) \surj H^2_{nis}(X, \sK^M_{2, (X,D)})$.
\end{thm}

\begin{remk}\label{remk:Surj-FF-aff}
The map $cyc_{(X,D)}$ turns out to be injective as well if $X$ is affine. 
A proof of this using completely different type of argument 
will appear in \cite{BiK}.
\end{remk}

Let us now assume that $X$ is a resolution of singularities of a normal
surface $Y$ and let $U$ denote the regular locus of $Y$. Then we can
use \thmref{thm:Surj-FF} to obtain the following finer result which
fully answers Question~\ref{ques:KES-1} for a special class of
surfaces. 

\begin{thm}\label{thm:Surj-FF-Res}
Let $k$ be any field and let $X$ be a resolution of singularities of a normal
surface $Y$. Let $U$ denote the regular locus of $Y$. 
Then the cycle class map $\CH_0(X,D) \to H^2_{nis}(X, \sK^M_{2, (X,D)})$
induces an isomorphism
\[
{\underset{D}\varprojlim} \ \CH_0(X,D) \xrightarrow{\simeq}
{\underset{D}\varprojlim} \ H^2_{nis}(X, \sK^M_{2, (X,D)}),
\]
where the limits are taken over all effective divisors on $X$ with
support outside $U$.
\end{thm}

\subsection{Localization sequence for Chow groups with modulus}
\label{sect:LCM}
Since the introduction of the Chow groups with modulus, various
authors have been trying to prove several properties of these
Chow groups which are analogous to the well known properties
of Bloch's higher Chow groups.
It was shown in \cite{KP-1} recently that the Chow groups with
modulus satisfy projective bundle and blow-up formulas.
It was however not known if the localization sequence for Bloch's higher
Chow groups is true for Chow groups with modulus.
We use \thmref{thm:Surj-FF} to show that the Chow groups with modulus
do not admit such a localization sequence. 
In fact, we show that even the localization sequence for the 
ordinary Chow groups (in the sense of \cite{Fulton}) does
not admit extension to Chow groups with modulus.
Answering this question was another motivation of this note.

Let $m \ge 2$ be any integer and let $D$ denote the Cartier divisor 
$\Spec({k[t]}/{(t^m)})$ inside $\Spec(k[t])$. For any
$Y \in \Sch/k$, let us denote the Cartier divisor $Y \times D \inj
Y \times \A^1_k$ by $D$ itself.

\begin{thm}\label{thm:Local-fail}
Let $k$ be an algebraically closed field of characteristic zero with
infinite transcendence degree over $\Q$.
Let $Y$ be a connected projective curve over $k$ of positive genus. Then for 
any inclusion $i: \{P\} \inj Y$ of a closed point, the sequence
\[
\CH_0(\{P\} \times \A^1_k, D) \xrightarrow{i_*} \CH_0(Y \times \A^1_k, D) 
\xrightarrow{j^*}
\CH_0(Y \setminus \{P\} \times \A^1_k, D) \to 0
\]
is not exact. 

In particular, the localization sequence for Bloch's higher Chow groups
does not extend to the Chow groups with modulus,
even for a closed pair of smooth schemes.
\end{thm}

The proof of this negative result is based on \thmref{thm:Surj-FF}
and the following two vanishing theorems of independent interest.

\begin{thm}\label{thm:Affine-additive}
Let $k$ be any field and let $Y$ be any non-singular affine scheme over $k$
of dimension $d \ge 1$. Then $\CH_0(Y \times \A^1_k, D) = 0$.
\end{thm}

\begin{thm}\label{thm:Affine-gen}
Let $k$ be an algebraic closure of a finite field and let $X$ be a smooth 
affine scheme over $k$ of dimension $d \ge 3$. 
Then for any effective Cartier divisor $D \inj X$, we have $\CH_0(X, D) = 0$.
Assuming $D_{\rm red}$ is a normal crossing divisor, we also have
$H^d_{nis}(X, \sK^M_{d, (X,D)}) = 0$.
\end{thm}

\begin{remk}\label{remk:Affine-ques}
\thmref{thm:Affine-gen} implies that the analogue of Question~\ref{ques:KES-1}
has a positive solution for affine schemes over $k$ of dimension at least 
three if $D_{\rm red}$ is a normal crossing divisor.
\end{remk}

\begin{remk}\label{remk:Affine-gen-rem}
The assertion of \thmref{thm:Affine-gen} is true also for $d =2$ and will
appear in \cite{BiK}. The proof in this note does show
at least that $\CH_0(X,D)_{\Q} = 0$ even if $X$ is a surface.  

On the other hand, it is easily seen using the
surjection $\CH_0(X, D) \surj \CH_0(X)$ that $d \ge 2$ is a necessary condition
for the vanishing of $\CH_0(X,D)$.
\end{remk}

\subsection{Additive Chow groups and reciprocity functors}
\label{sect:REC}
The reciprocity functors $T(\sM_1, \cdots , \sM_r)$ were introduced by Ivorra 
and R{\"u}lling \cite{IR}. These reciprocity functors are expected to describe
the ordinary as well as the additive higher Chow groups of 0-cycles
for smooth projective schemes over a field.   
In this direction, it was shown by Ivorra and R{\"u}lling 
(see \cite[Corollary~5.2.5]{IR}) that for a smooth projective scheme $X$ of 
dimension $d$ over a field $k$, there is an isomorphism
$T(\G^{\times r}_m, {\CH_0(X)})(k) \simeq \CH^{d+r}(X,r)$. 
They also show that $T(\G_a, \CH_0(\Spec(k)))(k) \simeq \CH_0(\A^1_k, D_2)$ if
${\rm char}(k) = 0$, where $D_2 = \Spec({k[t]}/{(t^2)})$. 
This was a verification of a special case of the general expectation that
$T(\G_a, \CH_0(X))(k)$ should be isomorphic to the additive Chow group
$\CH_0(X\times \A^1_k, D_2)$ if $X$ is a smooth projective scheme over $k$.
However, combining Theorems~\ref{thm:Local-fail} and ~\ref{thm:Affine-additive}
with \cite[Theorem~1.1]{IR-1}, we prove:

\begin{cor}\label{cor:Rec-functor}
Let $k$ be an algebraically closed field of characteristic zero with
infinite transcendence degree over $\Q$.
Let $Y$ be a connected projective curve over $k$ of positive genus. 
Then $\CH_0(Y\times \A^1_k, D_2)$ can not be described in terms of
the reciprocity functors.
\end{cor}

\vskip .3cm

\subsection{Outline of proofs}
We recall the definitions of Chow groups with modulus in \S~\ref{sect:Recoll}.
We then use the Thomason-Trobaugh spectral sequence to relate the
cohomology of the sheaf $\sK^M_{2,(X,D)}$ with the relative $K$-groups.
We first prove an analogue of \thmref{thm:Surj-FF} for curves in 
\S~\ref{sect:CCM} and deduce it for surfaces using 
\lemref{lem:Cycle-class-0}. The proof of \thmref{thm:Surj-FF} is 
completed using some results of \cite{KS} and \thmref{thm:Surj-FF-Res} 
proven by using a combination of  \thmref{thm:Surj-FF} and 
an explicit formula for the Chow group of 0-cycles on normal surfaces
from \cite{KSri}.

We prove \thmref{thm:Affine-additive} by first reducing to the case of
curves. This case is achieved with the help of an algebraic version of
a sort of containment lemma. We prove \thmref{thm:Local-fail} as a
combination of Theorems~\ref{thm:Surj-FF} and ~\ref{thm:Affine-additive}.
This reduces the problem to understanding a map of cohomology groups
of the relative $K$-theory sheaves of nilpotent ideals. This in turn
can be written as an explicit map of $k$-vector spaces, where $k$ is the 
ground field. \thmref{thm:Affine-gen} is proven by reducing
to the case of affine surfaces and empty Cartier divisor using some Bertini
theorems.

\section{Recollection of Chow group with modulus and 
relative $K$-theory}\label{sect:Recoll}
We fix a field $k$ and let $\Sch/k$ denote the category of quasi-projective
schemes over $k$. Let $\Sm/k$ denote the full subcategory of $\Sch/k$
consisting of non-singular (regular) schemes. 
Given $X \in \Sch/k$, we shall write
$X_{\rm sing}$ and $X_{\rm reg}$ for the closed and open subschemes of $X$,
where $X_{\rm red}$ is singular and regular, respectively. 
In this text, a {\sl curve} will mean
an equi-dimensional quasi-projective scheme over $k$ of dimension one.
For a curve $C$, the scheme $C^N$ will often denote the normalization of
$C_{\rm red}$.
Given a closed immersion $Y \inj X$ in $\Sch/k$, we let $|Y|$ denote the
support of $Y$ with the reduced induced closed subscheme structure.

For $X \in \Sch/k$, let $K(X)$ and $G(X)$ denote the $K$-theory spectra
of perfect complexes and coherent sheaves on $X$, respectively.
For a closed subscheme $Y \inj X$, let $K(X,Y)$ denote the homotopy fiber
of the restriction map $K(X) \to K(Y)$. 
For a sheaf $\sF$ on the small Zariski (resp. Nisnevich) site
of $X$, let $H^*_{zar}(X,\sF)$ (resp. $H^*_{nis}(X,\sF)$)
denote the cohomology groups of $\sF$.
A cohomology group in this text without mention of the underlying site 
will indicate the  Zariski cohomology.

\subsection{Thomason-Trobaugh spectral sequence for $K$-theory with 
support and relative $K$-theory}\label{sect:TTSS}
Given a scheme $X$ and a closed subscheme $Y \inj X$, let $K^Y(X)$ denote the
homotopy fiber of the restriction map of spectra $K(X) \to K(X \setminus Y)$.
Let $\sK_{i, (X,Y)}$ denote the Zariski sheaf on $X$ whose stalk at a point
$x \in X$ is the relative group $K_i(\sO_{X,x}, \sO_{Y,x})$ for $i \in \Z$.
Given a closed point $x \in X_{\rm reg} \setminus Y$, the spectrum
$K^{\{x\}}(Y)$ is contractible and hence there are natural maps of 
spectra
\begin{equation}\label{eqn:Surj-FF-3}
K(k(x)) \to K^{\{x\}}(X) \to K(X,D) \to K(X).
\end{equation}
In particular, there is a commutative diagram of
Thomason-Trobaugh spectral sequences (see \cite[Corollary~10.5]{TT})
\begin{equation}\label{eqn:Surj-FF-4}
\xymatrix@C.8pc{
E^{p,q}_{2,x} = H^p_{\{x\}}(X, \sK_{q, X}) \ar@{=>}[r] \ar[d] &
K^{\{x\}}_{q-p}(X) \ar[d] \\
E^{p,q}_{2, (X,Y)} = H^p(X, \sK_{q, (X,Y)}) \ar@{=>}[r] \ar[d] &
K_{q-p}(X,Y) \ar[d] \\
E^{p,q}_{2, X} = H^p(X, \sK_{q, X}) \ar@{=>}[r] & K_{q-p}(X)}
\end{equation} 
which is valid even when the Zariski cohomology is replaced by the Nisnevich
cohomology.

\begin{lem}\label{lem:K0-mod}
Given a modulus pair $(X,D)$ of dimension two over $k$, there is a short exact 
sequence
\begin{equation}\label{eqn:KO-mod-0}
0 \to H^2_{\sC}(X, \sK_{2, (X,D)}) \to K_0(X,D) \to \Pic(X,D) \to 0
\end{equation}
where $\sC$ is Zariski or Nisnevich cohomology.
In particular, the map $H^2_{zar}(X, \sK_{2, (X,D)}) \to 
H^2_{nis}(X, \sK_{2, (X,D)})$ is an isomorphism. 
\end{lem}
\begin{proof}
Let $\sC$ denote either the Zariski or the Nisnevich cohomology.
Since the $\sC$-cohomological dimension of $X$ is two, the
strongly convergent spectral sequence
$E^{p,q}_2 = H^p_{\sC}(X, \sK_{q, (X,D)}) \Rightarrow K_{q-p}(X,D)$ 
with differential $d_r: E^{p,q}_r \to E^{p+r, q+r-1}_r$ gives us an exact sequence
\begin{equation}\label{eqn:KO-mod-1}
H^0_{\sC}(X, \sK_{1, (X,D)}) \xrightarrow{d^{0,1}_2} 
H^2_{\sC}(X, \sK_{2, (X,D)}) \to K_0(X,D) \to H^1_{\sC}(X, \sK_{1, (X,D)}) 
\to 0.
\end{equation}

It follows from Hilbert's theorem 90 that the map 
$H^1_{zar}(X, \sK_{1, (X,D)}) \to H^1_{nis}(X, \sK_{1, (X,D)})$ is an
isomorphism and it follows from \cite[Lemma~2.1]{SV} that 
$H^1_{zar}(X, \sK_{1, (X,D)}) \xrightarrow{\simeq} \Pic(X,D)$.
We are thus left with showing that $d^{0,1}_2 = 0$. We prove this for
the Zariski cohomology as the same argument applies in the Nisnevich case.

Applying the above spectral sequence for $K_1(X,D)$, $d^{0,1}_2 =0$ is
equivalent to the assertion that the map
$K_1(X,D) \to H^0(X, \sK_{1, (X,D)})$ is surjective.
To prove this, we let $f \in H^0(X, \sK_{1, (X,D)})$. This
is equivalent to a regular map $f: X \to \G_m$ such that $f|_D =1$ 
and hence a commutative diagram with exact rows
%be any element
%and consider the resulting commutative square
%\begin{equation}\label{eqn:KO-mod-2}
%\xymatrix@C1pc{
%D \ar[r] \ar[d] & X \ar[d]^{f} \\
%\{1\} \ar[r] & \G_m.}
%\end{equation}
%This yields a commutative diagram
\begin{equation}\label{eqn:KO-mod-2}
\xymatrix@C.8pc{
0 \ar[r] & K_1(\G_m, \{1\}) \ar[r] \ar[d]^{f^*} & K_1(\G_m) \ar[r] \ar[d]^{f^*}
& K_1(\{1\}) \ar[r] \ar[d]^{f^*} & 0 \\
& K_1(X,D) \ar[r] \ar[d] & K_1(X) \ar[r] \ar[d]^{\delta} & K_1(D) \ar[d] & \\
0 \ar[r] & H^0(X, \sK_{1, (X,D)}) \ar[r] & H^0(X, \sK_{1, X}) \ar[r] &
H^0(D, \sK_{1, D}). &}
\end{equation}

If we let $\G_m = \Spec(k[t^{\pm 1}])$, then one can check (as is well known)
that $\delta \circ f^*([t]) = f$. Since $t \in K_1(\G_m, \{1\})$, we
see that $f^*([t]) \in K_1(X,D)$ and $\delta \circ f^*(t)$ dies in
$H^0(D, \sK_{1, D})$. Hence, it must lie in $H^0(X, \sK_{1, (X,D)})$.
It follows that the map $K_1(X,D) \to H^0(X, \sK_{1, (X,D)})$ is surjective.
\end{proof}

\begin{remk}\label{remk:Zar-Nis}
The isomorphism $H^2_{zar}(X, \sK_{2, (X,D)}) \xrightarrow{\simeq} 
H^2_{nis}(X, \sK_{2, (X,D)})$ was shown earlier by Kato and Saito 
(see \cite[Proposition~9.9]{KS}) by a different method.
\end{remk}

\subsection{Chow groups of 0-cycles with modulus}
\label{sec:0-cyc-mod}
We recall the definition of the Chow group of 0-cycles with modulus
(see \cite[\S~2]{BS} or \cite[\S~2]{KP-1}).

Let $X$ be a non-singular scheme of pure dimension $d$ and let 
$D \subsetneq X$ be an effective 
Cartier divisor on $X$. We shall call such a pair $(X,D)$ of a
non-singular scheme and an effective Cartier divisor, a $d$-dimensional
{\sl modulus pair}.
Let $\sZ_0(X, D)$ denote the free abelian group
on the closed points in $X \setminus D$. Let $C \inj X \times \P^1_k$
be a closed irreducible curve satisfying
\begin{enumerate}
\item
$C$ is not contained in $X \times \{0, 1, \infty\}$.
\item
If $\nu: C^N \to X \times \P^1_k$ denotes the composite map from the
normalization of $C$, then one has an inequality of Weil divisors on $C^N$:
\[
\nu^*(D \times \P^1_k) \le \nu^*(X \times \{1\}).
\]
\end{enumerate}

We call such curves admissible. Let $\sZ_1(X, D)$ denote the
free abelian group on admissible curves and let $\sR_0(X, D)$ denote
the image of the boundary map 
$(\partial_{0} - \partial_{\infty}): \sZ_1(X, D) \to \sZ_0(X, D)$.
The Chow group of 0-cycles on $X$ with modulus $D$ is defined as
the quotient
\[
\CH_0(X, D) := \frac{\sZ_0(X, D)}{\sR_0(X, D)}.
\]

To relate this definition of $\CH_0(X, D)$ with the one given by
Kerz and Saito \cite{KeS}, let $\pi_C:C^N \to C$ denote 
the normalization of an integral curve $C \inj X$ which is not a component
of $D$. Let $A_{C|D}$ and $A_{C^N|D}$ denote the semi-local rings
of $C$ and $C^N$ at the supports of $C \cap D$ and $\pi^{-1}_C (C \cap D)$,
respectively. Let $\sR'_0(X,D)$ denote the subgroup of
$\sZ_0(X, D)$ given by the image
\begin{equation}\label{eqn:Rel-Chow}
{\underset{C \not\subset D}\coprod} K_1(A_{C^N|D}, I_D)  
\xrightarrow{{\rm div}} \sZ_0(X, D).
\end{equation}

Note that the surjectivity of the map $K_2(A_{C^N|D}) \to K_2(\pi^*_C(D))$
implies that 
\begin{equation}\label{eqn:Rel-Chow-0}
K_1(A_{C^N|D}, I_D) = {\rm Ker}(K_1(A_{C^N|D}) \to K_1(\pi^*_C(D)) =
{\underset{U}\varinjlim} \ {\rm Ker}(\sO(U)^{\times} \to 
\sO({\pi^*_C(D)})^{\times}),
\end{equation}
where $U$ ranges over all open subschemes of $C^N$ containing $\pi^*_C(D)$.

One can then check as in the classical case (e.g., see \cite[Theorem~3.3]{BS})
that there is a canonical isomorphism 
\begin{equation}\label{eqn:equal}
\frac{\sZ_0(X, D)}{\sR'_0(X, D)} \xrightarrow{\simeq} \CH_0(X, D).
\end{equation}

\section{The cycle class map}\label{sect:CCM}
Let $(X,D)$ be a 2-dimensional modulus pair.  
In this section, we construct the cycle class map $\CH_0(X,D) \to
H^2(X, \sK_{2, (X,D)})$ and prove Theorems~\ref{thm:Surj-FF} and
~\ref{thm:Surj-FF-Res}. More generally, we assume $X$ is either a curve
or a surface and let $P \in X \setminus D$ be a closed point. Let
$X_P$ denote the spectrum of the local ring $\sO_{X,P}$. Assume $d= 1,2$.
It follows from ~\eqref{eqn:Surj-FF-3} and ~\eqref{eqn:Surj-FF-4} that
there is a commutative diagram

\begin{equation}\label{eqn:CC-0-0}
\xymatrix@C1pc{
H^0(\{P\}, \sK_{0, \{P\}}) \ar[d]_{\simeq} \ar[r] & K_0(\{P\}) \ar[d]^{\simeq} \\
H^d_{\{P\}}(X, \sK_{d, X}) \ar[d]\ar[r] &
K^{\{P\}}_0(X) \ar[d] \\
H^d(X, \sK_{d, (X,D)}) \ar[r] & K_0(X,D),}
\end{equation}
where the top vertical arrow on the left is an isomorphism by excision and the
Gersten resolution for $\sK_{d, X_P}$ and the one on the right is an 
isomorphism by the localization sequence for $K$-theory.
We define the cycle class map 
\begin{equation}\label{eqn:CCmap-def}
cyc_{(X,D)}: \sZ_0(X,D) \to H^d(X, \sK_{d, (X,D)})
\end{equation} 
by letting $cyc_{(X,D)}([P])$ be the image of $1 \in H^0(\{P\}, \sK_{0, \{P\}})
\simeq \Z$ under the composite vertical arrow on the left in 
~\eqref{eqn:CC-0-0} and extending it linearly on all of $\sZ_0(X,D)$.
To show that this map kills rational equivalences, we first consider the case of
curves.

\begin{lem}\label{lem:Bloch-curve}
Let $(C,D) $ be an 1-dimensional modulus pair.
Then the map $cyc_{(C,D)}$ induces isomorphisms
\[
cyc_{(C,D)}: \CH_0(C,D) \xrightarrow{\simeq} H^1_{zar}(C, \sK_{1, (C,D)}) 
\xrightarrow{\simeq} H^1_{nis}(C, \sK_{1, (C,D)}) \xrightarrow{\simeq}
\Pic(C,D) \xrightarrow{\simeq} K_0(C,D).
\]
\end{lem}
\begin{proof}
For any reduced closed subset $S \subsetneq C$ such that $S \cap D = \emptyset$
and any open subset $U \subseteq X$,
we have the localization fiber sequence of spectra
\[
K(S \cap U) \to K(U) \to K(U \setminus S).
\]
Taking the filtered colimit over closed subsets $S$ as above under the
inclusion, we get a short exact sequence of Zariski sheaves
\begin{equation}\label{eqn:curve-BF}
0 \to \sK_{1, (C,D)} \to j_*(\sK_{1, (C_D, D)}) \to
{\underset{P \notin D} \coprod} \ (i_P)_*(K_0(k(P))) \to 0
\end{equation}
on $C$, where $C_D$ is the spectrum of the semi-local ring $A_{C|D}$ of $C$ at 
$|D|$ and $j: C_D \inj C$ is the inclusion map.
This yields the cycle class map
\begin{equation}\label{eqn:curve-BF-0}
cyc_{(C,D)}: {\underset{P \notin D} \coprod} \ \Z \to H^1(C, \sK_{1, (C,D)}).
\end{equation}

To show that this induces an isomorphism $\CH_0(C,D) \to H^1(C, \sK_{1, (C,D)})$,
we first claim that $j_*(\sK_{1, (C_D, D)})$ is an acyclic Zariski sheaf. 
To prove this claim, it suffices to show that if $U \inj C$ is open and $U_D$
is the spectrum of the semi-local ring of $U$ at $|U \cap D|$, then
$H^i(U_D, \sK_{1, (U_D,D)}) = 0$ for $i \ge 1$.
But this is immediate from the exact sequence 
\[
0 \to  \sK_{1, (U_D,D)} \to  \sK_{1, U_D} \to  \sK_{1, U \cap D} \to 0
\]
and the fact that $U_D$ is a semi-local scheme.

It follows from the above claim that ~\eqref{eqn:curve-BF} is an 
%not a flasque resolution
acyclic resolution of $\sK_{1, (C,D)}$ and in particular, there
is an exact sequence
\[
K_1(A_{C|D}, I_D) \xrightarrow{\rm div} {\underset{P \notin D} \coprod} 
\ \Z \to H^1_{zar}(C, \sK_{1, (C,D)}) \to 0.
\]
This implies by ~\eqref{eqn:equal} that ~\eqref{eqn:curve-BF-0}
induces an isomorphism $\CH_0(C,D) \xrightarrow{\simeq} 
H^1_{zar}(C, \sK_{1, (C,D)})$.

The isomorphism of the natural map $H^1_{zar}(C, \sK_{1, (C,D)}) 
\to H^1_{nis}(C, \sK_{1, (C,D)})$ follows easily from Hilbert's theorem 90.

We now consider the commutative diagram of homotopy fiber sequences
\[
\xymatrix@C1pc{
{\underset{P \notin D}\coprod} K(k(P)) \ar[r] & K(C) \ar[r] \ar[d] & 
K(A_{C|D}) \ar[d] \\
& K({A_{C|D}}/I) \ar@{=}[r] & K({A_{C|D}}/I).}  
\]
This yields a homotopy fiber sequence
\[
{\underset{P \notin D}\coprod} K(k(P)) \to K(C,D) \to K(A_{C|D},I)
\]
and in particular, an exact sequence
\[
K_1(A_{C|D},I) \xrightarrow{\partial} \sZ_0(C, D) \to K_0(C,D) \to 0
\]
and we conclude from this that 
\[
{\rm Coker}(\partial) = \CH_0(C, D) \xrightarrow{\simeq} K_0(C, D).
\]
Finally, the isomorphism $H^1_{zar}(C, \sK_{1, (C,D)}) 
\xrightarrow{\simeq} \Pic(C,D)$ follows from \cite[Lemma~2.1]{SV}.
\end{proof}

\begin{lem}\label{lem:Cycle-class-0}
Let $(X,D)$ be a 2-dimensional modulus pair and let $f:C \to X$ be a finite
map, where $C$ is a non-singular curve such that $f^*(D)$ is a proper
closed subscheme of $C$. Then there is a commutative diagram
\begin{equation}\label{eqn:Cycle-class-1}
\xymatrix@C2pc{
\sZ_0(C, f^*(D)) \ar[r]^{cyc_{(C, f^*(D))}} \ar[d]_{f_*} & 
H^1(C, \sK_{1, (C, f^*(D))}) \ar[d]^{f_*} \\
\sZ_0(X, D) \ar[r]^<<<<<<{cyc_{(X, D)}} & 
H^2(X, \sK_{2, (X,D)})}
\end{equation}
where $f_*$ on the left is the push-forward map.
\end{lem}
\begin{proof}
We set $E = f^*(D)$. Since $\iota_X: D \inj X$ and $\iota_C: E \inj C$ are 
Cartier divisors, it follows that ${\rm Tor}^i_{\sO_X}(\sO_D, f_*(\sO_C)) = 0$ 
for $i > 0$. In particular, there is a commutative diagram
\begin{equation}\label{eqn:Cycle-class-2}
\xymatrix@C1pc{
K(C) \ar[r]^{f_*} \ar[d]_{\iota^*_C} & K(X) \ar[d]^{\iota^*_X} \\
K(E) \ar[r]_{f_*} & K(D).}
\end{equation}

As this diagram makes sense for any open $U \inj X$
and is functorial for restriction to open subsets, we see that
~\eqref{eqn:Cycle-class-2} is in fact a diagram of presheaves of spectra on 
$X_{zar}$. 

If we consider the homotopy cofibers of the horizontal arrows
in ~\eqref{eqn:Cycle-class-2}, we obtain a commutative diagram of homotopy 
cofiber sequences of presheaves of spectra on $X_{zar}$. 
Taking the long homotopy groups exact sequences, we obtain
the associated diagram of the long exact sequences of
the presheaves of homotopy groups. The exactness of the sheafification
functor yields a commutative diagram of the long exact sequences
of the sheaves of homotopy groups corresponding to 
~\eqref{eqn:Cycle-class-2}.

Let $\wt{K}(X\setminus C)$ and $\wt{K}(D \setminus E)$ denote the homotopy
cofibers of the top and bottom horizontal arrows in ~\eqref{eqn:Cycle-class-2},
respectively. Let $\wt{\sK}_{i, X\setminus C}$ denote the Zariski sheaf on
$X$ associated to the presheaf of homotopy groups $U \mapsto 
\pi_i(\wt{K}(U\setminus C))$. Defining $\wt{\sK}_{i, D\setminus E}$ in a
similar way, we get a commutative diagram of the long exact sequences
 
\begin{equation}\label{eqn:Cycle-class-3}
\xymatrix@C.8pc{
\cdots \ar[r] & \wt{\sK}_{3,{X\setminus C}} \ar[r] \ar[d] & 
f_*(\sK_{2, C})  \ar[r] \ar[d] & 
\sK_{2, X} \ar[r] \ar[d] &  \wt{\sK}_{2, {X\setminus C}} \ar[r] 
\ar[d] & f_*(\sK_{1, C}) \ar[r] \ar[d] & \cdots \\
\cdots \ar[r] & \wt{\sK}_{3, D \setminus E} \ar[r] & 
f_*(\sK_{2, E}) \ar[r] & \sK_{2,D} \ar[r] & 
\wt{\sK}_{2, D \setminus E} \ar[r] &
f_*(\sK_{1, E}) \ar[r] & \cdots .}
\end{equation}

If $\ov{C}$ is the image of $f: C \to X$, then we have a factorization 
$K(C) \to G(\ov{C}) \to K(X)$ (see \cite[Proposition~5.12 (i)]{Srinivas})
and this shows that there is a factorization
$\sK_{i, X} \to \wt{\sK}_{i, {X\setminus C}} \to j_*(\sK_{i, {X\setminus \ov{C}}}) \to
j_*(K_i(k(X)))$, where $j: X \setminus \ov{C} \inj X$ is the inclusion. 
The Gersten resolution says that the composite map is
injective. It follows that the map $\sK_{i, X} \to \wt{\sK}_{i, {X\setminus C}}$
is injective. Since the map $f_*(\sK_{i, C}) \to f_*(\sK_{i, E})$ is surjective
for $i \le 2$,
the above diagram refines to a commutative diagram of short exact sequences
\begin{equation}\label{eqn:Cycle-class-4}
\xymatrix@C.8pc{ 
0 \ar[r] & \sK_{2, X} \ar[r] \ar[d] & \wt{\sK}_{2,{X\setminus C}} \ar[r] 
\ar[d]^{\phi} & f_*(\sK_{1, C}) \ar[r] \ar[d] & 0 \\
0 \ar[r] & \sK_{2,D} \ar[r] &  \wt{\sK}_{2, D \setminus E} \ar[r] &
f_*(\sK_{1, E}) \ar[r] & 0.}
\end{equation}

Set $\wt{\sK}_{2, (X, D)} = {\rm Ker}(\sK_{2, X} \to \sK_{2,D})$.
Since the vertical arrows on the left and the right ends in
~\eqref{eqn:Cycle-class-4} are surjective, the
middle arrow is also surjective and there is a short exact sequence of the 
kernel sheaves
\begin{equation}\label{eqn:BF-surf-4-shv} 
0 \to \wt{\sK}_{2, (X, D)} \to {\rm Ker}(\phi) \to 
f_*(\sK_{1, (C, E)}) \to 0.
\end{equation}

Considering the long exact cohomology sequences with and without support
and observing that $H^i(C, f_*(\sK_{1, (C, E)})) \simeq H^i(C, \sK_{1, (C,E)})$
(the higher direct images of $\sK_{1, (C, E)}$ vanish as one can easily check),
we get a commutative diagram
\begin{equation}\label{eqn:Cycle-class-5}
\xymatrix@C.8pc{ 
{{\underset{Q \in \Sigma_P}\coprod} \ \Z} \ar[r]^<<<{\simeq} \ar[d]_{f_*} &
H^1_{\Sigma_P}(C, \sK_{1, C}) \ar[r]^<<<{\simeq} \ar[d] & 
H^1_{\Sigma_P}(C, \sK_{1, (C, E)}) \ar[r] \ar[d] & 
H^1(C, \sK_{1, (C,E)}) \ar[d] \\
\Z \ar[r]^<<<<{\simeq} & 
H^2_{\{P\}}(X, \sK_{2, X}) \ar[r]^<<{\simeq} & 
H^2_{\{P\}}(X, \wt{\sK}_{2, (X, D)}) \ar[r] & 
H^2(X, \wt{\sK}_{2, (X,D)})}
\end{equation} 
for any closed point $P \in X \setminus D$ and $\Sigma_P = f^{-1}(P)$.
It is well known that the left-most vertical map is the push-forward map.
Since the map $\sK_{2, (X, D)} \to \wt{\sK}_{2, (X,D)}$ is surjective whose
kernel is supported on $D$, the map 
$H^2(X, \sK_{2, (X,D)}) \to H^2(X, \wt{\sK}_{2, (X,D)})$ 
is an isomorphism. This immediately yields ~\eqref{eqn:Cycle-class-1}.
\end{proof}

\subsection{Proof of \thmref{thm:Surj-FF}}
In view of \lemref{lem:K0-mod}, the proof of \thmref{thm:Surj-FF} is
reduced to showing that the cycle class map $cyc_{(X,D)}: \sZ_0(X,D) \to 
H^2(X, \sK_{2, (X,D)})$ constructed in ~\eqref{eqn:CCmap-def}
kills the group of rational equivalences $\sR'_0(X,D)$ 
(see ~\eqref{eqn:equal}) and is surjective.
So, let us take an integral curve $C \inj X$ which is not contained in $D$
and let $f: C^N \to X$ denote the induced map from the normalization of $C$.
Letting $E = f^*(D)$ and 
$g \in {\rm Ker}(\sO^{\times}_{C^N} \surj \sO^{\times}_E)$,
we need to show that
$cyc_{(X,D)} \circ f_*({\rm div}(g)) = 0$.  For this, we consider the
diagram
\begin{equation}\label{eqn:Cycle-class-6}
\xymatrix@C2pc{ 
\sR'_0(C^N,E) \ar[r] \ar[d]_{f_*} & \sZ_0(C^N,E) \ar[r]^<<<<{cyc_{(C^N,E)}} 
\ar[d]^{f_*} &  H^1(C^N, \sK_{1, (C^N,E)}) \ar[d]^{f_*} \\
\sR'_0(X,D) \ar[r] & \sZ_0(X,D) \ar[r]^<<<<<{cyc_{(X,D)}} &  
H^2(X, \sK_{2, (X,D)})}
\end{equation}
in which the left square commutes by \cite[Proposition~2.10]{KP-1}
and the right square commutes by \lemref{lem:Cycle-class-0}.
Since the composite horizontal map on the top is zero by 
\lemref{lem:Bloch-curve}, it follows that 
$cyc_{(X,D)} \circ f_*({\rm div}(g)) = f_* \circ cyc_{(C^N,E)}({\rm div}(g)) = 0$.
The surjectivity of $cyc_{(X,D)}$ now follows from 
\lemref{lem:Cycle-class-0}, the isomorphism $\sK^M_{2,(X,D)} \xrightarrow{\simeq}
\wt{\sK}_{2,(X,D)}$, the diagram~\eqref{eqn:CC-0-0} and \cite[Theorem~2.5]{KS}.
$\hfill\square$

\subsection{Proof of \thmref{thm:Surj-FF-Res}}
Let $\pi:X\to Y$ be a resolution of singularities of a normal surface over
any field $k$. We set $U = Y_{\rm reg}$ and $C(U) = {\underset{D}\varprojlim} \ 
\CH_0(X, D)$, where the limit is taken over all effective Cartier
divisors on $X$ with support outside $U$. Let $E \inj X$ denote the
reduced exceptional divisor.
If $D \subsetneq X$ is an effective Cartier divisor with
support $|D| \subseteq E$, then $mE - D$ must be an effective Cartier
divisor some $m \gg 1$. This implies that the canonical maps
$C(U) \to {\underset{m}\varprojlim} \ \CH_0(X, mE)$ and
${\underset{D}\varprojlim} \ H^2(X, \sK_{2,(X,D)}) \to 
{\underset{m}\varprojlim} \ H^2(X, \sK_{2,(X,mE)})$
are isomorphisms. 

Let $\CH_0(Y)$ denote the Chow group of 0-cycles on $Y$ in the sense of
\cite{LW} and let $S \inj Y$ denote the singular locus of $Y$ with
reduced subscheme structure. We then have a commutative diagram
\begin{equation}\label{eqn:0-C-map-I-0}
\xymatrix@C2pc{
\CH_0(Y) \ar[r]^<<<<<<{cyc_{(Y,mS)}} \ar[d]_{\pi^*} & H^2(Y, \sK_{2,(Y,mS)}) 
\ar[d]^{\pi^*} \ar[dr]^{\simeq} & \\
\CH_0(X,mE) \ar[r]^{cyc_{(X, mE)}} \ar[d] & 
H^2(X, \sK_{2,(X,mE)})
\ar[d] & H^2(Y, \sK_{2,Y}) \ar[dl]^{\pi^*} \\
\CH_0(X) \ar[r]^{cyc_X} & H^2(X, \sK_{2,X}). &}
\end{equation}

The map $cyc_{(Y,mS)}$ is defined exactly like $cyc_{(X,mE)}$ and is
an isomorphism by \cite[Proposition~3.1]{Kr-1}. The natural map
$H^2(Y, \sK_{2,(Y,mS)}) \to H^2(Y, \sK_{2,Y})$ is an isomorphism also by
 \cite[Proposition~3.1]{Kr-1}.
The map $\pi^*: \CH_0(Y) \to \CH_0(X,mE)$  is induced by the
identity map $\pi^*: \sZ_0(U) \to \sZ_0(X, mE)$. 

To show that it preserves rational
equivalences, let $C \inj Y$ be an integral
curve not meeting $S$ and let $h \in k(C)^{\times}$.
Let $\Gamma_h \inj C \times \P^1 \inj Y \times \P^1$ be the graph of the
function $h: C \to \P^1$. It is then clear that $\Gamma_h \cap (S
\times \P^1) = \emptyset$. In particular, $\pi^{-1}(\Gamma_h) \cap
(E \times \P^1) = \emptyset$. This shows that $[\Gamma_h] \in
\sZ_1(X, mE)$ is an admissible 1-cycle such that
\[
\pi^*({\rm div}(h)) = \pi^*([h^*(0)] - [h^*(\infty)]) =
\pi^*(\partial_{0}([\Gamma_h]) - \partial_{\infty}([\Gamma_h]))
= (\partial_{0} - \partial_{\infty})([\Gamma_h]).
\]
This shows that $\pi^*({\rm div}(h)) \subset \sR_0(X, mE)$
and yields the pull-back $\pi^*: \CH_0(Y) \surj \CH_0(X, mE)$.
All other maps in ~\eqref{eqn:0-C-map-I-0} are naturally defined and all 
are surjective.

If we let $F^2K_0(X, mE)$ denote the image of the map
$cyc_{(X, mE)}: \CH_0(X,mE) \to K_0(X,mE)$, then it follows from
\thmref{thm:Surj-FF} and \lemref{lem:K0-mod} that 
$F^2K_0(X, mE) \to H^2(X, \sK_{2,(X,mE)})$ is an isomorphism.
We now apply \cite[Theorem~1.1]{KSri} to conclude that the map
$H^2(Y, \sK_{2,(Y,mS)}) \to  H^2(X, \sK_{2,(X,mE)})$ is an isomorphism
for all sufficiently large $m$. It follows that all arrows in
the upper square of  ~\eqref{eqn:0-C-map-I-0} are isomorphisms for
all sufficiently large $m$. 
In particular, the map $cyc_{(X,mE)}: \CH_0(X,mE) \to H^2(X, \sK_{2,(X,mE)})$ is 
an isomorphism for all sufficiently large $m$ and hence the map
$C(U) \to {\underset{m}\varprojlim} \ H^2(X, \sK_{2,(X,mE)})$
is an isomorphism.
$\hfill\square$

\section{Vanishing theorems and failure of localization}
\label{sect:LF}
Let $k$ be a field and consider the effective Cartier divisor 
$D = \Spec({k[t]}/{t^m})$ on $\A^1_k = \Spec(k[t])$.
Given $X \in \Sch/k$, let us denote the effective Cartier divisor
$X \times D \inj X \times \A^1_k$ by $D$ itself.
We shall prove \thmref{thm:Affine-additive} using the following algebraic 
result.

\begin{lem}\label{lem:Alg-lem}
Let $A$ be the coordinate ring of a smooth affine curve over $k$
and let $\fm$ be a maximal ideal of $A[t]$ which contains the
ideal $(t-a)$, where $a \in k^{\times}$. Then we can find a prime
ideal $\fp$ of height one in $A[t]$ such that the following hold.
\begin{enumerate}
\item
$\fp \subsetneq \fm$.
\item
${A[t]}/{\fp}$ is smooth.
\item
${\fm}/{\fp}$ is a principal ideal.
\item
$\fp + (t) = A[t]$.
\end{enumerate}
\end{lem}
\begin{proof}
Consider the maximal ideal $\fm' = \fm \cap A$ of $A$. Since $A$ is a 
Dedekind domain, we can write $\fm' = (f_1,f_2)$. But this implies
using our hypothesis that $\fm = (t-a, f_1, f_2) = (a^{-1}t-1, f_1, f_2)$.
In case $f_1 = f_2$, we take $\fp = (t-a)$ which clearly does the job.
So we assume that $f_1 \neq f_2$. 

Since $A_{\fm'}$ is a discrete valuation ring, $\fm'A_{\fm'}$ is a principal
ideal. In particular, there is an element $f \in A$ such that $f \notin \fm'$
and $\fm' A_f$ is principal. As $f \notin \fm'$, we have $(f) + \fm' = A$,
and this gives us an identity $\alpha f - \alpha_1 f_1 - \alpha_2 f_2 -1 = 0$
in $A$. Setting $g = \alpha f$, we see that $\fm'A_g$ is also a 
principal ideal. Furthermore, we have
\begin{equation}\label{eqn:Alg-lem-0}
ga^{-1}t- 1 = g(a^{-1}t-1) + g -1   = 
g(a^{-1}t-1) + \alpha_1 f_1 + \alpha_2 f_2 \in \fm.
\end{equation}

If we set $\fp = (ga^{-1}t-1) \subsetneq A[t]$, we have just shown that
$\fp \subsetneq \fm$. Since ${A[t]}/{\fp} \simeq A_{g}$ 
and hence
\[
\frac{\fm}{\fp} \simeq \frac{\fm A_g[t]}{\fp A_g[t]} \simeq 
\frac{(-g^{-1}(\alpha_1 f_1 + \alpha_2 f_2), f_1, f_2) A_g[t]}{\fp A_g[t]}
\simeq \frac{(f_1, f_2) A_g[t]}{\fp A_g[t]} \simeq \fm'A_g,
\]
we see that (2) and (3) are satisfied. The item (4) is clear.
This proves the lemma.
\end{proof}

\subsection{Proof of \thmref{thm:Affine-additive}}
We can assume that $Y$ is connected. We set $X = Y \times \A^1_k$ and
$U = Y \times \G_m$. Let $p: X \to \A^1_k$ and $q: X \to Y$ denote the
projection maps. Let $P \in U$ be a closed point and set
$P_1 = p(P)$ and $P_2 = q(P)$. Then $P_1 \in \G_m$ and $P_2 \in Y$ are
closed points as well. 

We can find a non-singular curve $\iota: C \inj Y$ containing $P_2$
(see \cite[Theorem~1]{AK} when $k$ is infinite and \cite[Theorem~1.1]{Poonen}
when $k$ is finite). 
It follows from \cite[Proposition~2.10]{KP-1} that there is a 
push-forward map $\iota_*: \CH_0(C \times \A^1_k, D) \to
\CH_0(Y \times \A^1_k, D)$ such that the class $[P] \in 
\CH_0(Y \times \A^1_k, D)$ lies in the image of this map. 
We can therefore assume that $Y$ is a curve.

Now $P$ defines a unique closed point
$P' \in X_{k(P)}$ such that $P = \pi(P')$, where
$\pi: \Spec(k(P)) \to \Spec(k)$ is the finite map. This gives $[P] =
\pi_*([P'])$  under the push-forward map \
$\pi_*: \CH_0(X_{k(P)}, D) \to
\CH_0(X,D)$ (see \cite[Proposition~2.10]{KP-1}). It suffices
therefore to show that the class $[P'] \in \CH_0(X_{k(P)},D)$ 
dies. We can thus assume that $P_1 \in \G_m(k)$.

We can now apply \lemref{lem:Alg-lem} to get a smooth affine curve
$i: C \inj X $ which is a closed subset of $X$ containing $P$ such that 
$C \cap (Y \times D) = \emptyset$ and $P \in C$ is a principal Cartier divisor. 
In particular, the class $[P] \in \CH_0(C)$ is zero.
On the other hand, the condition $C \cap (Y \times D) = \emptyset$ implies
that the inclusion $\sZ_0(C) \inj \sZ_0(X, D)$ defines a 
push-forward map $i_*: \CH_0(C) \to \CH_0(X,D)$ 
(see \cite[Corollary~ 2.11]{KP-1}) such that $i_*([P]) = [P] \in \CH_0(X,D)$.
It follows that $[P] = 0$. This proves that $\CH_0(X, D) = 0$.
The second part of the theorem now follows from \thmref{thm:Surj-FF}.
$\hfill\square$

\vskip .3cm

As an immediate consequence of Theorems~\ref{thm:Surj-FF} and 
~\ref{thm:Affine-additive}, we get

\begin{cor}\label{cor:Aff-add-1}
Given a non-singular affine curve $Y$ over a field $k$, we have
\[
K_0(Y \times \A^1_k, D) \xrightarrow{\simeq} \Pic(Y \times \A^1_k, D).
\]
\end{cor}

\begin{remk}\label{remk:Aff-proj-add}
Theorem~\ref{thm:Affine-additive} is known to fail when $d =0$
(see \cite{BE-2}).
\end{remk}

\subsection{Proof of \thmref{thm:Affine-gen}}
Let the pair $(X, D)$ be as in \thmref{thm:Affine-gen} and let 
$x \in X \setminus D$ be a closed point. We can assume that $X$ is connected.
We claim that there is a smooth affine closed subscheme
$\iota: Y \inj X$ of dimension $d-1$ such that $Y \cap D = \emptyset$
and $x \in Y$.  

To prove the claim, let $A$ denote the coordinate ring of $X$ and let
$I \inj A$ denote the defining ideal of $D$. Let $\fm \inj A$ denote the
maximal ideal corresponding to $x \in X$. Our assumption implies that
there exist elements $a \in \fm^2$ and $b \in I$ such that $a-b =1$.
We can now apply \cite[Theorems~1.3, 1.4]{Swan} to conclude that
for general $a' \in \fm^2$, the ring ${A}/{(a -a'b)}$ is integral and 
smooth. Setting $f = a- a'b$, we see that $f \in \fm$ and
$f-1 = a-a'b-1 = b-a'b = b(1-a') \in I$. This shows that $Y := \Spec({A}/{(f)})$
satisfies our requirement.

Using the above claim and \cite[Corollary~2.11]{KP-1}, we get a push-forward
map $\iota_*: \CH_0(Y) \to \CH_0(X,D)$ whose image contains the cycle class
$[x]$. The desired vanishing now follows because one knows that
$\CH_0(Y) = 0$ (see e.g., \cite[Theorem~6.4.1]{KS-1}).

To prove the second assertion of the theorem, we first notice that
for a closed point $x \in X \setminus D$, we have natural maps 
\[
K_0(k(x)) \xrightarrow{\simeq} H^d_{\{x\}}(X, \sK^M_{d,(X,D)}) \to 
H^d_{zar}(X,  \sK^M_{d,(X,D)}) \to H^d_{nis}(X,  \sK^M_{d,(X,D)}). 
\]
Setting $cyc_{(X,D)}([x])$ to be the image of $1 \in K_0(k(x))$ under the 
composite map, we get a cycle class map 
$cyc_{(X,D)}: \sZ_0(X,D) \to H^d_{nis}(X,  \sK^M_{d,(X,D)})$.

If $D_{\rm red}$ has normal crossings, then it follows from
\cite[Definition~3.4.1, Proposition~3.5]{RS} that $cyc_{(X,D)}$ has 
a factorization
$\CH_0(X,D) \to \H^{2d}_{nis}(X, \sZ(d)_{X|D}) \to H^d_{nis}(X,  \sK^M_{d,(X,D)})$,
where $\sZ(d)_{X|D}$ is the sheaf of cycle complexes
$U \mapsto \sZ^d(U|D, 2d-\bullet)$ on $X_{nis}$.
Moreover, it follows from \cite[Theorem~2.5]{KS} that the map
$cyc_{(X,D)}: \CH_0(X,D) \to H^d_{nis}(X,  \sK^M_{d,(X,D)})$ is surjective.
The vanishing of $H^d_{nis}(X,  \sK^M_{d,(X,D)})$ now follows from the
first part of the theorem.
$\hfill\square$

\vskip .4cm

\subsection{Proof of \thmref{thm:Local-fail}}\label{sect:fail-M}
In view of \thmref{thm:Affine-additive}, the theorem is equivalent to
the assertion that the push-forward map 
$\CH_0(\{P\} \times \A^1_k, D) \xrightarrow{i_*} \CH_0(Y \times \A^1_k, D)$
is not surjective.
Note that the composite map $\CH_0(\{P\} \times \A^1_k, D) \xrightarrow{i_*} 
\CH_0(Y \times \A^1_k, D) \xrightarrow{\pi_*} \CH_0(\A^1_k,D)$
is an isomorphism, where $\pi:Y \to \Spec(k)$ is the structure map.
In particular, $i_*$ is split injective.
Our aim is to show that this is not surjective.

We set $X = Y \times \A^1_k, V = Y \setminus \{P\}$, $U = V \times \A^1_k$
and $Z = \{P\} \times \A^1_k$. For any $W \in \Sch/k$, we shall
write $W \times D$ as $W_D$ in this proof. In view of \thmref{thm:Surj-FF},
it suffices to show that the composite map
$\CH_0(Z, D) \xrightarrow{i_*} \CH_0(X, D)
\xrightarrow{cyc_{(X,D)}} H^2(X, \sK_{2, (X,D)})$ is not surjective.

Let $\sH^P_{Y_D}$ denote the exact category of coherent sheaves on $Y_D$ which
have cohomological dimension at most one and which are supported on 
$\{P\} \times D$ so that there is a commutative diagram of the
fiber sequences of spectra (see \cite[Theorem~9.1]{Srinivas})
\begin{equation}\label{eqn:Fail-0}
\xymatrix@C.8pc{
K(Z) \ar[r] \ar[d] & K(X) \ar[r] \ar[d] & K(U) \ar[d] \\
K(\sH^P_{Y_D}) \ar[r] &  K(Y_D) \ar[r] &  K(V_D).}
\end{equation}

As in the proof of \lemref{lem:Cycle-class-0}, this diagram canonically
extends to a commutative diagram of presheaves of spectra. 
Let $\sK^P_{i, Y_D}$ denote
the Zariski sheaf on $Z$ associated to the presheaf of homotopy groups
$W \mapsto \pi_i(K(\sH^P_{Y_D \cap W}))$. 
Sheafifying the associated presheaves of
homotopy groups and arguing as in the proof of \lemref{lem:Cycle-class-0},
we obtain the commutative diagrams of short exact sequence of Zariski sheaves

\begin{equation}\label{eqn:Quillen-*-3}
\xymatrix@C1pc{
& 0 \ar[d] & 0 \ar[d] & 0 \ar[d]   & \\
0 \ar[r] &  \wt{\sK}_{2,(X,D)} \ar[r]  \ar[d] & j_*(\wt{\sK}_{2,(U,D)}) 
\ar[r]  \ar[d] & i_*({\sK}^P_{1, Z}) \ar[r] \ar[d] & 0 \\
0 \ar[r] & \sK_{2,X} \ar[r]  \ar[d] & j_*(\wt{\sK}_{2,U}) \ar[r]  \ar[d] &
\sK_{1, Z} \ar[r] \ar[d] & 0 \\
0 \ar[r] & \sK_{2, Y_D} \ar[r] \ar[d] & j_*(\sK_{1,V_D}) \ar[r] \ar[d] &
\sK^P_{1, Y_D} \ar[r] \ar[d] & 0 \\
& 0 & 0 & 0 &}
\end{equation}   
and

\begin{equation}\label{eqn:Quillen-*-4}
\xymatrix@C1pc{
0 \ar[r] & \sK_{1, (Z,D)} \ar[r] \ar[d] & \sK_{1,Z} \ar@{=}[d] \ar[r] & 
\sK_{1, {\{P\}}_D} \ar[r] \ar[d] & 0 \\
0 \ar[r] &  {\sK}^P_{1, Z} \ar[r] & \sK_{1,Z} \ar[r] & \sK^P_{1, Y_D} \ar[r] & 0}
\end{equation} 

These diagrams together give rise to a commutative diagram of exact sequences

\begin{equation}\label{eqn:Quillen-*-5}
\xymatrix@C1pc{
0 \ar[r] & H^0(Z, \sK_{1,Z}) \ar[r]^<<{\iota^*_{(Z,D)}} \ar[d] & 
H^0({\{P\}}_D, \sK_{1,{\{P\}}_D}) \ar[r]^>>>{\partial_Z} \ar[d] &
H^1(Z, \sK_{1, (Z,D)}) \ar[r] \ar[d]^{i_*} & 0 \\
0 \ar[r] & H^1(X, \sK_{2,X}) \ar[r]^{\iota^*_{(X,D)}} & H^1(Y_D, \sK_{2,Y_D}) 
\ar[r]_{\partial_X}  & H^2(X, \sK_{2, (X,D)}) \ar[r]  & 0.}
\end{equation}

The maps $\partial_Z$ and $\partial_X$ are surjective because
$H^1(Z, \sK_{1, Z}) \simeq \CH_0(Z) = 0 =
\CH_2(X) \simeq H^2(X, \sK_{2, X})$.
The homotopy invariance of
$K$-theory tells us that the composite map
$H^0(Z, \sK_{1,Z}) \xrightarrow{\iota^*_{(Z,D)}}
H^0({\{P\}}_D, \sK_{1,{\{P\}}_D)}) \to H^0(\{P\}, \sK_{1,\{P\}})$
is an isomorphism.
We claim that the composite map
$H^1(X, \sK_{2,X}) \xrightarrow{\iota^*_{(X,D)}} 
H^1(Y_D, \sK_{2,Y_D}) \to H^1(Y, \sK_{2,Y})$ is also an
isomorphism.

We have a commutative diagram
\begin{equation}\label{eqn:Quillen-*-6}
\xymatrix@C1pc{
K_1(Y) \ar[r] \ar[d] &  H^0(Y, \sK_{1,Y}) \ar[d] \\
K_1(X) \ar[r] & H^0(X, \sK_{1,X}),}
\end{equation}
where the vertical arrows are isomorphisms and the horizontal arrows are
split surjections.
This implies that the induced pull-back map $SK_1(Y) \to SK_1(X)$
is an isomorphism.
We now have a commutative diagram
\begin{equation}\label{eqn:Quillen-*-7}
\xymatrix@C1pc{
SK_1(Y) \ar[r] \ar[d] &  H^1(Y, \sK_{2,Y}) \ar[d] \\
SK_1(X) \ar[r] & H^1(X, \sK_{2,X}),}
\end{equation}
where the top horizontal arrow is an isomorphism and the bottom horizontal
arrow is surjective (see \cite[Lemma~2.3]{KSri}). We have shown above that
the left vertical arrow is an isomorphism. This implies that the 
right vertical arrow is surjective.
On the other hand, it is split injective via the 0-section embedding.
Hence, it is an isomorphism. This proves the claim.

The claim shows that the first horizontal arrows from left in both rows of
~\eqref{eqn:Quillen-*-5} are split injective.
Combining this with Lemmas~\ref{lem:Bloch-curve} and ~\ref{lem:Cycle-class-0}, 
we can identify $i_*: \CH_0(Z,D) \to H^2(X, \sK_{2, (X, D)})$ as the map
\begin{equation}\label{eqn:Quillen-*-8}
i_*: K_1(\{P\} \times D, \{P\} \times \{0\}) \to
H^1(Y_D, \sK_{2, (Y_D, Y)}).
\end{equation}
Using \cite[Corollary~4.2]{KSri}, this map is same as the map of
$\Q$-vector spaces
\begin{equation}\label{eqn:Quillen-*-9} 
i_*: I \to H^1(Y_D, \frac{\Omega^1_{{(Y_D,Y)}/{\Q}}}{d(I_Y)}),
\end{equation}
where $I$ is the ideal sheaf of $\Spec(k)$ inside $D$,
$I_Y = I \otimes_k \sO_Y$ and
$\Omega^1_{{(Y_D,Y)}/{\Q}} = {\rm Ker}(\Omega^1_{{Y_D}/{\Q}} \surj 
\Omega^1_{Y/{\Q}})$.
We are thus reduced to showing that this map of $\Q$-vector spaces is
not surjective. Notice that the assumption $m \ge 2$ implies that $I \neq 0$.

By \cite[Lemma~4.3]{KSri}, there is a short exact sequence
\[
0 \to \Omega^1_{k/{\Q}} \otimes_k I_Y \to
 \frac{\Omega^1_{{(Y_D,Y)}/{\Q}}}{d(I_Y)} \to 
\frac{\Omega^1_{{(Y_D,Y)}/{k}}}{d_k(I_Y)} \to 0.
\]

It is easy to check by local calculations that 
$\frac{\Omega^1_{{(Y_D,Y)}/{k}}}{d(I_Y)} \simeq \Omega^1_{Y/{\Q}} \otimes_k
d_k(I)$, where $d_k: I \to \Omega^1_{D/k}$ is the $k$-derivation.
In particular, the above exact sequence can be written as
\begin{equation}\label{eqn:Quillen-*-10} 
0 \to (I \otimes_k \Omega^1_{k/{\Q}}) \otimes_k\sO_Y \to 
\sK_{2, (Y_D, Y)} \to d_k(I) \otimes_k \Omega^1_{Y/k} \to 0.
\end{equation}

Taking the associated long exact cohomology sequence, we get a
commutative diagram
\begin{equation}\label{eqn:Quillen-*-11} 
\xymatrix@C1pc{
& & I \ \ \  \ar[d]_{i_*} \ar[dr]^{d_k} &  \\
d_k(I) \otimes_k H^0(Y, \Omega^1_{Y/k}) \ar[r]^<<\partial & 
(I \otimes_k \Omega^1_{k/{\Q}}) \otimes_k H^1(Y, \sO_Y) \ar[r] &
H^1(Y_D, \sK_{2, (Y_D, Y)}) \ar[r] & d_k(I) \ar[r] & 0}
\end{equation}
with the bottom sequence exact.

It is straightforward to check that $d_k$ is an isomorphism.
On the other hand, as $k$ has infinite transcendence degree over $\Q$ and
$Y$ has positive genus, we see that $\partial$ is a map of $k$-vector
spaces whose source is finite dimensional but the target is
infinite dimensional. This shows that there is a
split exact sequence
\begin{equation}\label{eqn:Quillen-*-12}
0 \to \frac{(I \otimes_k \Omega^1_{k/{\Q}}) \otimes_k H^1(Y, \sO_Y)}
{d_k(I) \otimes_k H^0(Y, \Omega^1_{Y/k})} \to 
H^1(Y_D, \sK_{2, (Y_D, Y)}) \to  d_k(I) \to 0
\end{equation}
such that the first term is an infinite dimensional $k$-vector space
and the composite map $I \xrightarrow{i_*} H^1(Y_D, \sK_{2, (Y_D, Y)}) \to 
d_k(I)$ is an isomorphism. In particular, the cokernel of $i_*$ is
an infinite dimensional $k$-vector space. This finishes the proof of 
\thmref{thm:Local-fail}.
$\hfill\square$

\vskip .5cm

\noindent\emph{Acknowledgements.}
The author would like to thank Jinhyun Park for his questions
related to connections between additive Chow groups and reciprocity
functors that led to \S~\ref{sect:REC}. The author would like to
thank Marc Levine and Federico Binda for invitation to the university
of Duisburg-Essen at Essen in April 2015, where this paper was revised.
The author would also like to thank the referee for carefully reading the
paper and suggesting valuable improvements.

\end{document}